\documentclass[11pt]{amsart}

\usepackage{amsthm}
\usepackage{amsmath}
\usepackage{amssymb}
\usepackage{tikz}
\usepackage{wrapfig}
\usepackage[all]{xy}
\usepackage{graphicx}
\usetikzlibrary{patterns}
\usetikzlibrary{calc}
\usepackage{enumerate}
\usepackage{multirow}
\usepackage{url}
\usepackage{tabularx,booktabs}

\usepackage{todonotes}

\newtheorem{theorem}{Theorem}[section]
\newtheorem{corollary}[theorem]{Corollary}
\newtheorem{lemma}[theorem]{Lemma}
\newtheorem{proposition}[theorem]{Proposition}

\newtheorem{question}[theorem]{Question}

\theoremstyle{remark}
\newtheorem{example}[theorem]{Example}
\newtheorem{remark}[theorem]{Remark}

\theoremstyle{definition}
\newtheorem{definition}[theorem]{Definition}

\newcommand{\RR}{\mathbb{R}}
\newcommand{\ZZ}{\mathbb{Z}}
\newcommand{\QQ}{\mathbb{Q}}
\newcommand{\NN}{\mathbb{N}}

\newcommand{\uu}{\textbf{u}}

\newcommand{\ee}{\textbf{e}}

\newcommand{\OO}{\textbf{0}}

\newcommand{\calN}{\mathcal{N}}

\newcommand{\conv}{\mathrm{conv}}

\newcommand{\PP}{\mathbb{P}}

\newcommand{\nvol}{\mathrm{nvol}}
\newcommand{\fcone}{\mathrm{fcone}}

\newcommand{\lin}{\mathrm{lin}}

\newcommand{\relint}{\mathrm{relint}}
 
\DeclareMathOperator{\chisel}{ch}
\newcommand{\cube}{C_n}

\begin{document}


\title{Smooth polytopes with negative Ehrhart coefficients}

\author{Federico Castillo}
\address[Federico Castillo]{405 Snow Hall, 1460 Jayhawk Blvd. Lawrence, KS 66045, USA}
\email{fcastillo@ku.edu}

\author{Fu Liu}
\address[Fu Liu]{One Shields Avenue, Department of Mathematics, University of California, Davis, CA 95616, USA}
\email{fuliu@math.ucdavis.edu}

\author{Benjamin Nill}
\address[Benjamin Nill]{Institut f\"ur Algebra und Geometrie, Otto-von-Guericke-Universit\"at Magdeburg, Geb\"aude 03, Universit\"atsplatz 2, 39106 Magdeburg, Germany}
\email{benjamin.nill@ovgu.de}

\author{Andreas Paffenholz}
\address[Andreas Paffenholz]{Fachbereich Mathematik, TU Darmstadt, Dolivostr. 15, 64293 Darmstadt, Germany}
\email{paffenholz@mathematik.tu-darmstadt.de}

\thanks{Fu Liu is partially supported by NSF grant DMS-1265702 and a grant from the Simons Foundation \#426756. Benjamin Nill is partially supported by DFG grant 314838170, GRK 2297 MathCoRe, and by Vetenskapsr{\aa}det grant NT:2014-3991.}

\begin{abstract}
We present examples of smooth lattice polytopes in dimensions $3$ and higher with the maximal possible number of negative Ehrhart coefficients. This answers a question by Bruns. We also discuss Berline-Vergne valuations as a useful tool in proving Ehrhart positivity results.
\end{abstract}

\maketitle

\section{Introduction}
\label{sec:introduction}

A {\em lattice polytope} is the convex hull of finitely many elements of $\ZZ^n$. An $n$-dimensional lattice polytope $P$ is called {\em smooth} (or {\em Delzant}) if each vertex is contained in precisely $n$ edges, and the primitive edge directions form a lattice basis of $\ZZ^n$. Equivalently, the normal fan of $P$ is a {\em unimodular fan}, i.e., each cone in the fan is spanned by a lattice basis. The importance of smooth polytopes stems from the fact that they correspond to very ample torus-invariant divisors on nonsingular projective toric varieties (as well as from their significance as moment polytopes of symplectic toric manifolds). We refer to \cite{CLS} for more background and motivation.

The study of invariants of lattice polytopes that are invariant under {\em unimodular transformations} (i.e., affine lattice automorphisms) leads inevitably to the study of the coefficients of their Ehrhart polynomials, see \cite{BK}. For this, we consider the function $i(P, t) := |tP \cap \ZZ^n|$ for $t \in \ZZ_{>0}$ that counts lattice points in dilates of $P$. In \cite{Ehrhart} Ehrhart proved that this function extends to a polynomial function $i(P,t)$ of degree $\dim(P)$, called the {\em Ehrhart polynomial} of $P$. It is known (we refer to the book \cite{BRbook}) that the leading coefficient equals the Euclidean volume of $P$, the second highest coefficient equals half of the boundary volume of $P$, and the constant coefficient equals $1$. In particular, these three Ehrhart coefficients are always positive rational numbers. However, starting from dimension 3, there are examples of Ehrhart polynomials of lattice polytopes where all the other coefficients can (even simultaneously) be negative, see \cite{Hibinegative}. Hence, it is interesting to ask when a lattice polytope has only positive Ehrhart coefficients. It turns out that quite a few families of polytopes have this property. The simplest example is the standard cube $[0,1]^n$, other families include standard crosspolytopes \cite[Exercise 4.61(b)]{enum}, zonotopes \cite{zonotopes}, the $y$-family of generalized permutohedra \cite{postnikov} which includes the previously already studied case of Stanley-Pitman polytopes \cite{stanley-pitman}, and cyclic polytopes and its generalizations \cite{highintegral}. It is worth mentioning that the proofs of these different families of polytopes having positive Ehrhart coefficients are all quite different. For example, Ehrhart positivity of a standard crosspolytope $\Diamond_n$ follows from the result that all roots of $i(\Diamond_n,t)$ have negative real parts, which is a consequence of the fact that the $h^*$-polynomial of $\Diamond_n$ has all roots on the unit circle on the complex plane. See \cite{EhrPosSurvey} for a survey with a general discussion on different techniques for attacking the problem of Ehrhart positivity. 

There are several actively researched properties of lattices polytopes, for example the \emph{Integer Decomposition Property} (IDP) or having a \emph{unimodular triangulation} (UT) \cite{bgbook}. It is well known that for any lattice polytope $P$ its large enough dilation $cP$ will have UT, and that UT implies IDP. However, note that dilating a polytope doesn't change the signs of the coefficients in its Ehrhart polynomial. Therefore, it is easy to give examples of lattice simplices in dimension 3 that show that Ehrhart positivity and UT/IDP are not directly related.

Let us note that the examples with many negative Ehrhart coefficients described in \cite{Hibinegative} are not smooth. In \cite[Question~7.1]{bruns}, Bruns posed the following problem:

\begin{question} Do the Ehrhart polynomials of smooth lattice polytopes have positive coefficients?
\label{ques:smooth}
\end{question}

As we will see, the answer to this question is no.
\begin{theorem}\label{main} In each dimension $n \ge 3$, there exists a smooth lattice polytope $P$ such that the $t^j$-coefficient of its Ehrhart polyomial $i(P,t)$ is negative for any $j=1, \ldots, n-2$.
\end{theorem}
 
The proof will be given in Section~\ref{sec:chiseling}. These examples are obtained by iteratively cutting off vertices of a dilated $3$-cube (following terminology in \cite{bruns} we will call this {\em chiseling}) and then taking the product with a dilated $(n-3)$-cube. 

We will also discuss one situation more closely. When cutting off all vertices of a dilated $n$-cube (without any further iterative chiseling), one needs at least dimension $7$ in order to get some negative linear coefficients. We can show that dimension $7$ is minimal in the following precise sense. Let us take the normal fan of the standard cube $[0,1]^n$ and make a stellar subdivision (also called star subdivision) of all of its full-dimensional cones (we refer again to \cite{CLS} for definitions). Let us denote this fan by $\calN_n$. 

\begin{proposition}\label{main-prop} For $n\le 6$, the Ehrhart polynomial of any smooth lattice polytope with normal fan $\calN_n$ has only positive coefficients. In each dimension $n \ge 7$, there exist smooth lattice polytopes with normal fan $\calN_n$ whose Ehrhart polynomials have negative linear coefficients.
\end{proposition}

As will be discussed in Section~\ref{sec:BV}, the main purpose of this result is to illustrate the method of Berline-Vergne valuations as a useful tool in proving positivity of Ehrhart coefficients by verifying the positivity of the so-called BV-$\alpha$-values. (This method was first developed in \cite{BValpha}, in which the authors reduced the problem of proving positivity of Ehrhart coefficients of all lattice generalized permutohedra to proving positivity of all BV-$\alpha$-values arising from regular permutohedra.) In particular, Lemma~\ref{counter} clarifies the relation between this stronger `alpha-positivity' of the normal fan and `Ehrhart-positivity' of the lattice polytope.

In toric geometry, Proposition~\ref{main-prop} has a natural interpretation.

\begin{corollary}
Let $X_n$ be the projective manifold obtained by blowing up $(\PP^1)^n$ in all it torus-invariant fixpoints. Then for $n\ge7$ there exists a very ample divisor $D$ on $X_n$ such that the Hilbert polynomial $k \mapsto h^0(X_n,k D)$ has a negative coefficient, while no such divisor exists for $n\le 6$.
\end{corollary}

Let us finish this introduction by discussing a weakening of Brun's original question. For this, let us recall that 
a lattice polytope is called {\em reflexive} if the origin is contained in its interior and every facet has lattice distance one from the origin. In fixed dimension $n$ there are only finitely many reflexive polytopes up to unimodular transformations. As the examples constructed in Theorem~\ref{main} are far from reflexive, it is a natural question, whether smooth reflexive polytopes have positive Ehrhart coefficients. Because of their correspondence to toric Fano manifolds, smooth reflexive polytopes were completely classified up to dimension $9$ \cite{Oebro, Database}. We used \texttt{polymake} \cite{Polymake} to check that up to dimension $8$ none of them has a negative Ehrhart coefficient. 
However, counterexamples come up in dimension $9$.

\begin{example}\label{ex:smoothreflexive}
Consider the lattice polytope $P$ defined by
\[ \begin{bmatrix}
\scriptstyle
1 & 0 & 0 & 0 & 0 & 0 & 0 & 0 & 0 \\
0 & 1 & 0 & 0 & 0 & 0 & 0 & 0 & 0 \\
0 & 0 & 1 & 0 & 0 & 0 & 0 & 0 & 0 \\
0 & 0 & 0 & 1 & 0 & 0 & 0 & 0 & 0 \\
0 & 0 & 0 & 0 & 1 & 0 & 0 & 0 & 0 \\
0 & 0 & 0 & 0 & 0 & 1 & 0 & 0 & 0 \\
0 & 0 & 0 & 0 & 0 & 0 & 1 & 0 & 0 \\
0 & 0 & 0 & 0 & 0 & 0 & 0 & 1 & 0 \\
0 & 0 & 0 & 0 & 0 & 0 & 0 & 0 & 1 \\
0 & 0 & 0 & 0 & 0 & 0 & 0 & 0 & -1 \\
-1 & -1 & -1 & -1 & 0 & 0 & 0 & 0 & 4 \\
0 & 0 & 0 & 0 & -1 & -1 & -1 & -1 & -4 
	\end{bmatrix} 
	\begin{bmatrix} x_1 \\ x_2 \\ x_3 \\ x_4 \\ x_5 \\ x_6 \\ x_7 \\ x_8 \\ x_9 \end{bmatrix}
	\le 
	\begin{bmatrix} 1 \\ 1 \\ 1 \\ 1 \\ 1 \\ 1 \\ 1 \\ 1 \\ 1 \end{bmatrix}
\]
We can check using \texttt{polymake} \cite{Polymake} that $P$ is smooth and reflexive, and its Ehrhart polynomial is 
\begin{align*}
i(P,t) =& \ 12477727/18144 t^9 + 12477727/4032 t^8 + 9074291/1512 t^7 +  630095/96 t^6 \\
& \ + 19058687/4320 t^5 +  117857/64 t^4 + 3838711/9072 t^3 + 11915/1008 t^2 \\
& \ -6673/630  t + 1,
\end{align*}
which has a negative linear coefficient. 
\end{example}

This motivates the following question:

\begin{question} Do there exist smooth reflexive polytopes in some dimension $n$ with the maximal possible number ($n-2$) of negative Ehrhart coefficients?
\end{question}

\section{Chiseling smooth lattice polytopes} 
\label{sec:chiseling}

Here, we will prove Theorem~\ref{main} and the second part of Proposition~\ref{main-prop}.

\subsection{Preliminaries}

Let us define the following lattice polytopes in $\RR^n$, here $\ee_1, \ldots, \ee_n$ denotes the standard basis:
\begin{itemize}
\item $C_n := [0,1]^n$ the $n$-dimensional {\em standard cube},
\item $\Delta_{n-1} := \conv(\ee_1,\cdots, \ee_n)$ the $(n-1)$-dimensional {\em standard simplex},
\item $S_n := \conv(\OO\cup \Delta_{n-1})$, the $n$-dimensional {\em unimodular simplex}.
\end{itemize}

Clearly, $\Delta_{n-1}$ is isomorphic to $S_{n-1}$ under a unimodular transformation. Let us note their well-known Ehrhart polynomials.

\begin{lemma}
$i(C_n,t)=(t+1)^n$,\; 
$i(\Delta_{n-1},t) = \binom{t+n-1}{n-1}$,\; 
$i(S_n,t) = \binom{t+n}{n}$.
\label{ehr-lemma}
\end{lemma}

\subsection{Iterative chiseling of dilates of the cube}

Let $P \subset \RR^n$ be an $n$-dimensional smooth lattice polytope. Let us choose a vertex $v$ of $P$ with primitive edge directions $u_1, \ldots, u_n$. Let us assume that there is a $b \in \ZZ_{>0}$ such that for all $i=1, \ldots, n$ the lattice point $v+b u_i$ is still in $P$ but not a vertex of $P$. Let us define $P'$ as the convex hull of all the vertices of $P$ except for $v$ together with the new vertices $v + b u_1, \ldots, v + b u_n$. We say that $P'$ is obtained from $P$ by {\em chiseling off the vertex $v$ at distance $b$}. It is straightfoward to check that $P'$ is still a smooth lattice polytope. We refer to \cite{bruns} for this terminology and more on chiseling smooth polytopes.

\begin{figure}[t]
\begin{center}
\begin{tikzpicture}[scale=.7]

\draw (0,0)--(0,3)--(2,3)--(3,2)--(3,0)--cycle;
\draw (4,0)--(7,0)--(7,3)--(4,3)--cycle;
\draw[dotted,color=black!90] (8,0)--(8,3)--(10,3)--(11,2)--(11,0)--cycle;
\draw (11,2)--(11,3)--(10,3)--cycle;
\draw[dotted,color=black!90] (12,0)--(12,3)--(14,3)--(15,2)--(15,0)--cycle;
\draw (15,2)--(14,3);
\node at (3.5,1.5){=};
\node at (7.5,1.5){$-$};
\node at (11.5,1.5){+};

\end{tikzpicture}
\end{center}
\caption{Inclusion-Exclusion for the operation of chiseling off a vertex at distance $b=1$ from $3C_2=[0,3]^2$.}
\label{fig:chisel}
\end{figure}
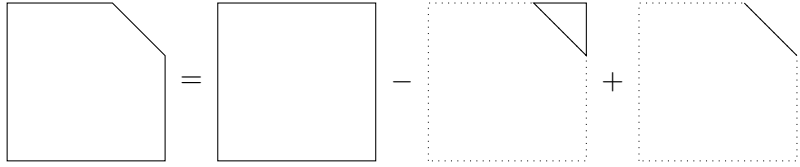

We compute the effect of chiseling a vertex $v$ of $P$ at distance $b$ on the Ehrhart polynomial. By inclusion-exclusion we obtain (e.g., see Figure~\ref{fig:chisel}):
\begin{align*}
 i(P',t)\ &=\  i(P, t) -\; i(b S_n, t)\; +\; i(b \Delta_{n-1}, t)\\
  &=\ i(P,t)\; -\; \binom{bt+n}{n}\; +\; \binom{b t+n-1}{n-1}\\ &=\ i(P,t)\; -\; \binom{bt+n-1}{n}\;.
\end{align*}

The \emph{lattice edge length} of an edge $e$ of $P$ is $\ell(e):=|e\cap\ZZ^n|+1$. Let $\ell_P\,:=\,\min(\ell(e)\mid e\text{ is an edge of }P)$. The \emph{full chiseling} $\chisel(P,b)$ of $P$ for an integer $b<\ell_P/2$ is obtained by chiseling all vertices of $P$ at distance $b$. Our choice of $b$ implies that the polytope $\chisel(P,b)$ is still a smooth polytope. Its normal fan is the fan obtained from the normal fan of $P$ by subdividing each maximal cone spanned by primitive generators $v_1, \ldots, v_n$ with $v_1+\cdots+v_n$. If $f_0$ denotes the number of vertices of $P$, then we obtain 
\begin{align}
 i(\chisel(P,b),t)\ &=\  i(P, t)\; -\; f_0 \binom{bt+n-1}{n}\;.\label{eq:chiselehrhart}
\end{align}

We can clearly iterate the chiseling process, as long as all edges still have lattice edge length at least $3$. For a lattice polytope $P$ and a sequence $B:=(b_1, \ldots, b_k)$ of positive integers we define a sequence $P_i$, $0\le i\le k$ of polytopes by $P_0:=P$ and $P_i:=\chisel(P_{i-1},b_i)$ for $1\le i\le k$. Then let $\chisel(P, B):=P_k$, the \emph{chiseling of $P$ by the sequence $B$}. Clearly, $\chisel(P,B)$ is smooth if $P$ is smooth and the lattice edge length of any edge of $P_{i-1}$ is at least $2b_i+1$ for $1\le i\le k$. Note that each step in this process multiplies the number of vertices by $n$, so if $P$ has $f_0$ vertices, then $P_i$ has $n^if_0$ vertices, for $0\le i\le k$. 

\subsection{Proof of the second part of Proposition~\ref{main-prop}}

Now, let $a, b$ be positive integers such that $a > 2b$. We define $Q_n(a,b) := \chisel(a C_n,b)$. Clearly, the normal fan of $Q_n(a,b)$ is $\calN_n$, the fan defined in the introduction. Using inclusion-exclusion we obtain as above:
\begin{align*}
i(Q_n(a,b),t) = (a t+1)^n - 2^n \binom{bt+n-1}{n}\,.
\end{align*}

Therefore, the linear coefficient of the Ehrhart polynomial of $Q_n(a,b)$ equals
\begin{align*}
a n - b \frac{2^n}{n}\,.
\end{align*}

In particular, for $a=5$ and $b=2$, elementary calculus shows that this value is strictly negative for $n \ge 7$. For instance, 
\begin{multline*}
  i(Q_7(5,2),t)\ =\ \frac{24608351}{315}t^7 \;+\;
  \frac{1640113}{15}t^6 \;+\; \frac{589345}t^5\\ \;+\;
  \frac{64729}{3}t^4 \;+\; \frac{182027}{45}t^3 \;+\;
  \frac{1729}{5}t^2 - \frac{11}{7}t \;+\; 1\,.
\end{multline*}

This proves the second part of Proposition~\ref{main-prop}.

\subsection{Proof of Theorem~\ref{main}}

For $k\ge 1$ let $B_k$ be the polytope
\begin{align*}
B_k\ :=\ \chisel(3^k\cdot C_3,(3^{k-1},3^{k-2},\ldots,3,1))\,,
\end{align*}
i.e., the polytope obtained by scaling the $3$-cube with $3^k$, then cutting all vertices at depth $3^{k-1}$, then cutting all vertices of the polytope obtained thereby at depth $3^{k-2}$, and continuing this process until in the last step we cut all vertices at depth $1$. This leads to a $3$-dimensional smooth lattice polytope with $8\cdot 3^k$ vertices and $4\cdot 3^k+2$ facets. Iterated application of \eqref{eq:chiselehrhart} lets us compute the Ehrhart polynomial of $B_k$. It is
\begin{align*}
i(B_k,t)\ &=\  (3^{k}t+1)^3\; -\; \sum_{j=0}^{k-1} 8\cdot 3^j\binom{3^{k-1-j}t+2}{3}\ =:\ q_3t^3+q_2t^2-q_1+1
\end{align*}
for some coefficients $q_1, q_2, q_3\in\QQ$. Note that $q_1$ is defined to be the negative of the first Ehrhart coefficient. 

By rearranging terms in the above formula, we obtain explicit formulas for $q_1, q_2, q_3:$
\begin{align*}
  q_1\ &=\ 3^{k-2}(8k-27)\\
  q_2\ &=\ 3^{k-1}(7\cdot 3^k+2)\\
  q_3\ &=\ \frac123^{k-2}(17\cdot 3^{2k}+1)\,.
\end{align*}
Note that $q_1, q_2,q_3>0$ for $k\ge 4$, so that the Ehrhart polynomial of $B_k$ has a negative linear coefficient for $k\ge 4$. 

Now let $\cube(a):=a\cdot C_n$ with Ehrhart polynomial $i(\cube(a),t)=(at+1)^n$ and define
\begin{align*}
  P^n(k,a)\ :=\ B_k\times \cube(a)\,.
\end{align*}
$P^n(k,a)$ is a $(3+n)$-dimensional smooth lattice polytope. The Ehrhart polynomial of a product of two lattice polytopes is the product of the Ehrhart polynomials of the factors, so the Ehrhart polynomial of $P^n(k,a)$ is
\begin{align*}
i(P^n(k,a),t)\ =\ (q_3t^3+q_2t^2-q_1t+1)(at+1)^n\ =:\ \sum_{i=0}^{n+3} \mu_it^i
\end{align*}
for coefficients $\mu_j\in\QQ$, $0\le j\le n+3$. By expanding this product we can write down explicit formulas for these coefficients:
\begin{align*}
  \mu_0\ &=\ 1\\
  \mu_1\ &=\ na-q_1\\
  \mu_j\ &=\
           \begin{alignedat}[t]{2}
             {n\choose j}a^j&\,-\, {n\choose j-1}a^{j-1}q_1\\&\,+\, {n\choose j-2}a^{j-2}q_2\,+\, {n\choose j-3}a^{j-3}q_3&\text{ for } 2\le j\le n+1
           \end{alignedat}
\\
  \mu_{n+2}\ &=\ a^{n-1}\left(aq_2+nq_3\right)\\
  \mu_{n+3}\ &=\ a^nq_3
\end{align*}

From now on, let $n \ge 1$ be fixed. We want to show that for $k$ large enough and for some $a$ depending on $k$ all $\mu_j$ for $1\le j\le n+1$ can be simultaneously negative. Let 
\begin{align}
  a \ := \ \left\lfloor \left. \frac{7}{n} k 3^{k-2} \right. \right\rfloor\,.\label{choice-a}
\end{align}
Clearly, for sufficiently large $k \ge 28$ we have that  
\begin{equation*} 
a\ \ge\ \frac{6}{n} k 3^{k-2}\ \ge\ 1\, ,
\end{equation*} 
and we may also assume that  
\begin{align}
q_1 \ &>\ n a,&  q_2\ \ &< \ 8 \cdot 3^{2k-1}, & q_3 \ &<\ 3^{3k}\,.\label{choice-k}
\end{align}
In particular, $\mu_1 < 0$. Let $2 \le j \le n+1$. Then
\begin{align*}
a^{3-j}\mu_j \ =\   {n\choose j}a^3\,-\, {n\choose j-1}a^2 q_1\,+\, {n\choose j-2}a q_2\,+\, {n\choose j-3} q_3 \,.
\end{align*}
By \eqref{choice-k}, we can strictly bound the right hand side from above by
\begin{align*}
a^{3-j}\mu_j \ &< \ \left({n\choose j}\,-\, {n\choose j-1}n \right) a^3\,+\, {n\choose j-2} 8 a \cdot 3^{2k-1}\,+\, {n\choose j-3} 3^{3k}\\
&=\ \left(\frac{n+1-j}{j}\,-\, n \right) {n\choose j-1} a^3\,+\, {n\choose j-2}8 a \cdot 3^{2k-1}\,+\, {n\choose j-3} 3^{3k}\,.
\end{align*}
As $\frac{n+1-j}{j}\,-\, n < 0$, the choice of $a$ in \eqref{choice-a} bounds this further from above by
\begin{multline*}
a^{3-j}\mu_j \ < \ \left(\frac{n+1-j}{j}\,-\, n \right) {n\choose j-1} \left(\frac{6}{n}\right)^3 k^3 3^{3k-6}\\ \,+\, {n\choose j-2} \frac{56}{n} k 3^{3k-3}\,+\, {n\choose j-3} 3^{3k}\,.
\end{multline*}
In other words, there exist positive numbers $c_1, c_2, c_3$ that only depend on $n$ and $j \in \{2, \ldots, n+1\}$ such that we can rewrite this inequality in the form  
\begin{align*}
a^{3-j}\mu_j \ < \ (-c_1 k^3 \,+\, c_2 k \,+\, c_3) 3^{3k}\,.
\end{align*}
For sufficiently large $k$ this expression will become negative. Hence, $\mu_j < 0$. This finishes the proof of Theorem~\ref{main}.\hfill$\qed$

\subsection{Examples}

Let us do a couple of examples. In dimension $3$ already the base polytope $B_k$ for sufficiently large $k$ provides an example of a smooth lattice polytope with negative first Ehrhart coefficient. For example, $B_4$ is a smooth lattice $3$-polytope with $648$ vertices, $972$ edges, and $326$ facets. Its Ehrhart polynomial is
\begin{align*}
  i(B_4,t)\ =\ 501921t^3\;+\; 15363t^2\;-\; 45t\;+\; 1\,.
\end{align*}
This is the first value of $k$ with a negative linear coefficient in the Ehrhart polynomial, as the Ehrhart polynomial of $B_3$ is 
\begin{align*}
  i(B_3,t)\ =\ 18591t^3\;+\; 1719t^2\;+\; 9t\;+\; 1\,.
\end{align*}

To get a $4$-dimensional example with values for $a$ and $k$ as in the proof, we can take $n:=1$, $k:=28$ and $a:=498205702352484$. Then 
\begin{align*}
&i(P^1(28,498205702352484),t)\ =\\  &\quad 5633398927928862087321748973638659814694960718075062244t^4\\&\quad\;+\;619688517319652881734980589359332452421773t^3\\&\quad \;-\;248254149429756452913678525969t^2\;-\;2541865828329t\;+\;1\,.
\end{align*}

Note that the choice of $a$ used in the proof above is a particular choice, while the construction also works for other values of $k$ and $a$ that often lead to considerably smaller lattice polytopes. For this, one should be aware that having chosen a specific $k$ one often cannot choose $a$ simply as in \eqref{choice-a} because it might be too large compared to a too small value of $k$. For instance, let $n=1$, $k=6$ and $a=730$ (here, $a$ is not chosen as in the proof, as $k=6$ would lead to $a=3402$, thus, $\mu_1 = 1701 > 0$). This gives a smooth lattice $4$-polytope with $11664$ vertices, $2920$ facets, and Ehrhart polynomial
\begin{align*}
  i(P^1(6,730),t)\ 
      \ &=\ 267104933370t^4 \;+\;1271473119t^3\;-\; 1215t^2\;-\;971t \;+\; 1\,.
\end{align*}
Its $h^*$-polynomial (the enumerator polynomial of the generating series of the Ehrhart polynomial) equals
\begin{align*}
  h^*(t)\ =\;\; &265833460008t^4\, \;+\; 2934339846011t^3 \;+\; 2941968690561t^2\\&+\; \;268376404299t \;+\; 1\,.
\end{align*}

Let us also give the Ehrhart polynomial of a $5$-dimensional smooth lattice polytope of the form $P^2(k,a)$, where we take $n=2$, $k=8$ and $a=8599$ (note again that these values are much smaller than the ones obtained in the proof above). We get
\begin{align*}
 i(P^2(8,8599),t)
                \;& = 19723429316570261841t^5 \;+\;12014689492982241t^4\\                    &\qquad\qquad\;-\;237422178t^3\;-\;289492130t^2\;-\;9775t\;+\;1\,.                    
\end{align*}
Its $h^*$-polynomial is 
  \begin{align*}
      h^*(t)\ =&\; 19711414627129339776t^5\;+\; 512689015334843195945t^4\\&+\;1301746334895061755914t^3\;+\; 512929309125860809290t^2\\&+\; 19735444005536319994t\;+\;1\,.
  \end{align*}
  Hence, the polytope  $P^2(8,8599)$ has $19735444005536320000$
lattice points, of which $24029378406980224$
 are on the boundary. 
  
  We used \texttt{normaliz}~\cite{Normaliz} and \texttt{polymake}~\cite{Polymake} to verify the computations for $P^1(6,730)$. The computation of the Ehrhart polynomial of $P^2(8,8599)$  did not finish within two days on a 64 core Intel Xeon E5. Already the polytope $P^{7}(19,2231332587)$, a $10$-dimensional Ehrhart negative example, has more than $4\cdot 10^{92}$ lattice points.
  \begin{remark}
    Our construction also works with other initial polytopes than a chiseling of the $3$-cube. For instance, we can use a prism over the smooth hexagon $H:=\conv(\pm(1,1), \pm(1,0), \pm(0,1))$. Let
    \begin{align*}
      H_k\ :=\  \chisel(3^k\cdot (H\times[0,1]),(3^{k-1},3^{k-2},\ldots,3,1))\,, 
    \end{align*}
    and $Q^n(k,a) := H_k \times C_n(a)$. Then $Q^1(5,457)$ is a smooth $4$-dimensional lattice polytope with $5832$ vertices and Ehrhart polynomial
    \begin{align*}
      i(Q^1(5,457),t)\ =\ 19125906543t^4\;+\;176889015t^3\;-\;648t^2\;-\;191t+1\,.
    \end{align*}
    We haven't found a $4$-dimensional smooth lattice polytope with negative linear and quadratic coefficient in the Ehrhart polynomial that has a smaller leading coefficient. Its $h^*$-polynomial is
    \begin{align*}
            h^*(t)\ =&\;18949017072t^4\;+\; 209854304999t^3\;+\; 210915640245t^2\\&+\; 19302794715t\;+\;1\,.   
    \end{align*}
    Also, $Q^3(9,46099)$ is a smooth $6$-dimensional lattice polytope with Ehrhart polynomial
    \begin{align*}
     i(Q^3(9,46099)\ =\;\;&  2178889417115552212024508181t^6\\&+\;331568043035736113553429t^5\;-\;340786031913009t^4\\&-\;615783337806158t^3\;-\;13464323277t^2\;-\;19167t\;+\;1\,.
    \end{align*} 
    It would be interesting to give lower bounds on volume, respectively, number of vertices or lattice points of smooth lattice polytopes with negative Ehrhart coefficients.
  \end{remark}

\section{Berline-Vergne valuations and BV-$\alpha$-values}
\label{sec:BV}

\subsection{Preliminaries}

We will show that any lattice polytope with normal fan $\calN_n$ (not just $Q_n(a,b)$) for $n\le 6$ has positive Ehrhart coefficients. While one could possibly exploit the specific nature of this class of examples to do some tedious optimization over their Ehrhart polynomials, we will use a more systematic approach that was developed in \cite{BValpha}. For this let us recall some notation. We denote by $\nvol(F)$ the volume of $F$ normalized with respect to the affine lattice of lattice points in the affine hull of $F$ (the normalization is such that $\nvol(S_n)=1/n!$). Given a polynomial $p(t)$ and $k \in \NN$, we define $[t^k] p(t)$ as the coefficient of the monomial $t^k$ in $p(t)$. 

\begin{definition}\label{defn:fcone}
  Suppose $P$ is a polyhedron and $F$ is a face. The \emph{feasible cone} of $P$ at $F$ is:
\begin{align*}
{\fcone(F,P)} = \left\{ \uu: x + \delta \uu \in P \hspace{5pt}\text{for sufficiently small $\delta$}\right\},
\end{align*}
where $x$ is any relative interior point of $F.$  (It can be checked that the definition is independent from the choice of $x.$) 
The \emph{pointed feasible cone} of $P$ at $F$ is 
\begin{align*}
\fcone^p(F,P) = \fcone(F,P)/\lin(F)\,.
\end{align*}
\end{definition}

\begin{remark}\label{rem:fcone}
  For convenience of the calculation we will carry out, we can extend the above definitions of $\fcone(F,P)$ and $\fcone^p(F,P)$ to any $F$ that is a convex subset of $P.$ Note that if $F$ is not a face of $P,$ the cone $\fcone^p(F,P)$ is not necessarily pointed. 
\end{remark}

Let us recall the \emph{BV-$\alpha$-valuation} due to Berline and Vergne. Here, $[ \ \cdot \ ]$ is the indicator function of sets.

\begin{theorem}[Berline, Vergne \cite{localformula}] There is a function $\Psi([C])$ on indicator functions of rational cones $C$ in $\RR^n$ with the following properties: 
\begin{enumerate}[(P1)]
  \item $\Psi( \cdot)$ is a valuation on the algebra of rational cones in $\RR^n$. %
  \item For a lattice polytope $P \subset \RR^n$, the following {\em local formula} (also called {\em McMullen's formula} \cite{mcmullen, BValpha}) holds:
    \begin{align*}
      \displaystyle 
      |P\cap \ZZ^n| = \sum_{F: \text{ a face of }P} \Psi([\fcone^p(F,P)])\, \nvol(F)\,.
    \end{align*}
  \item If a cone $C$ contains a line, then $\Psi([C])=0$.
\item $\Psi$ is invariant under the action of $O_n(\ZZ)$, the group of orthogonal unimodular transformations. More precisely, if $T$ is an orthogonal unimodular transformation, for any cone $C,$ we have $\Psi([C]) = \Psi([T(C)]).$ 
\end{enumerate}
\end{theorem}

  Let us define the {\em BV-$\alpha$-values} of faces $F$ of $P$ as 
  \begin{align*}
    \alpha(F,P) := \Psi([\fcone^p(F,P)])\,.
  \end{align*}

Note that as $\fcone^p(F,P)$ of $P$ at $F$ is dual to the normal cone of $P$ at $F$ the BV-$\alpha$-values of faces of $P$ depend only on its normal fan.

The local formula of property (P2) allows us to compute Ehrhart coefficients from the BV-$\alpha$-values of faces together with their normalized volumes:

\begin{corollary}[Theorem 3.1 of \cite{BValpha}] \label{thm:coefficient}
Let $P \subset \RR^n$ be a lattice polytope. Then 
\begin{align*}
[t^k] i(P,t) = \displaystyle \sum_{\dim(F)=k} \alpha(F,P) \ \nvol(F)\,.
\end{align*}
\end{corollary}

Our goal will be to show that $\alpha(F,Q_n(a,b)) > 0$ for all faces $F$ of $Q_n(a,b)$ if $n\le 6$. Clearly, this stronger positivity condition ($\alpha$-positivity) implies positivity of the coefficients of the Ehrhart polynomial $i(Q_n(a,b),t)$.

\subsection{Computing BV-$\alpha$-values}

Let us start by computing the BV-$\alpha$-values for the polytopes $C_n, \Delta_{n-1}$, and $S_n$.
\begin{lemma}
\begin{enumerate}
\item For any $k$-dimensional face $F$ of $C_n$, we have
  \begin{align*}
\displaystyle \alpha(F,C_n)=2^{k-n}\,.
  \end{align*}
\item For any $k$-dimensional face $F$ of $\Delta_{n-1}$, we have
  \begin{align*}
\alpha(F,\Delta_{n-1})=\dfrac{k!}{\binom{n}{k+1}} \cdot [t^k]\binom{t+n-1}{n-1}\,.
  \end{align*}
\item Let $F$ be a $k$-dimensional face of $S_n.$ Then
  \begin{align*}
 \alpha(F,S_n) = \begin{cases} 2^{k-n}, \quad & \text{if $\OO \in F$,} \\
      \displaystyle \dfrac{k! \cdot [t^k]\binom{t+n}{n} - \binom{n}{k} \ 2^{k-n}}{\binom{n}{k+1}}, & \text{otherwise.} \end{cases}
  \end{align*}
\end{enumerate}
\label{alpha-lemma}
\end{lemma}

\begin{proof}

(1) In the unit cube, all faces have normalized volume one. On top of that, the pointed feasible cones of $C_n$ at $k$-dimensional faces are in the same orbit under the action of $O_n(\ZZ)$, so by property (P4) they must all share the same $\alpha$-value, say $\alpha$. Applying Corollary~\ref{thm:coefficient} and Lemma~\ref{ehr-lemma}, a simple counting argument gives
$\binom{n}{k}2^{n-k}\alpha = \binom{n}{k}$.

(2) Each of the $\binom{n}{k+1}$ many $k$-dimensional faces of the standard simplex is a standard simplex itself, with volume $1/k!$. Similar to the case of $C_n$, all the pointed feasible cones of $\Delta_{n-1}$ at $k$-dimensional faces share the same $\alpha$-value $\alpha$. Applying Corollary~\ref{thm:coefficient} and Lemma~\ref{ehr-lemma}, we get $\displaystyle \binom{n}{k+1} \cdot 1/k! \cdot \alpha = [t^k]\binom{t+n-1}{n-1}$.

(3) There are two orbits of the pointed feasible cones at $k$-dimensional faces of $S_n$ under the action of $O_n(\ZZ)$. Any of the $\binom{n}{k}$ many $k$-dimensional face that contains the origin $\OO$ has the same pointed feasible cone as those we have already computed for the cube $C_n$. Thus their $\alpha$-value is $2^{k-n}$. The second orbit is given by the $k$-dimensional faces that do not contain the origin, there are $\binom{n}{k+1}$ of these. Similar to before, one can argue that all of them share the same $\alpha$-value, say $\alpha$. Finally, all $k$-dimensional faces of $S_n$ have the same volume $1/k!$. Putting this together, applying Corollary~\ref{thm:coefficient} and Lemma~\ref{ehr-lemma} yields: 
\begin{align*}
\binom{n}{k}\ 2^{k-n}\ \dfrac{1}{k!} + \binom{n}{k+1}\ \alpha \ \dfrac{1}{k!}=  [t^k]\binom{t+n}{n}.
\end{align*}
\end{proof}

Let $a > b$ be two positive integers. We define $P_n(a,b)$ to be the smooth lattice polytope obtained from $a C_n$ by chiseling off the origin at distance $b$.

\begin{lemma}
Let $F$ be a $k$-dimensional face of $P_n(a,b).$ Then
\begin{align*}
\alpha(F, P_n(a,b)) =
\begin{cases}
    \dfrac{\binom{n}{k}2^{k-n} - k!\cdot [t^k]\binom{t+n-1}{n} }{\binom{n}{k+1}}, \quad & \text{if $F$ is a face of $b \Delta_{n-1}$,} \\
  2^{k-n}, \quad & \text{otherwise}.
\end{cases}
\end{align*}
\label{crucial-lemma}
\end{lemma}

Let us remark that as BV-$\alpha$-values only depend on the normal cones, it is not surprising that the right side is independent of the precise values of $a$ and $b$.

\begin{proof}

We have the following inclusion-exclusion formula:
\begin{equation}
[P_n(a,b)] = [a C_n]-[b S_n]+[b \Delta_{n-1}].
  \label{equ:ie}
\end{equation}
See Figure \ref{fig:chisel} for a picture of $P_2(3,1)$.

There are two types of faces of $P_n(a,b)$, which we consider separately.

\begin{enumerate}[(i)]
  \item $F$ is a face of $b \Delta_{n-1}$: It follows from the definition of (pointed) feasible cones (Definition \ref{defn:fcone} and Remark \ref{rem:fcone}) and Formula \eqref{equ:ie} that
    \begin{align*}
  [\fcone^p(F, P_n(a,b))]\ =\ [\fcone^p(F, a C_n)]\;&-\;[\fcone^p(F, bS_n)]\\\;&+\;[\fcone^p(F, b\Delta_{n-1})].
    \end{align*}
Note that $\fcone^p(F, a C_n)$ is not pointed, i.e., contains a line. Hence, by Properties (P1) and (P3), we have
\begin{align*}
  \alpha(F, P_n(a,b))\ &=\ - \alpha(F, b S_n)+ \alpha(F,b \Delta_{n-1})\\
  &=\ - \alpha(1/b \cdot F, S_n)+ \alpha(1/b \cdot F, \Delta_{n-1}).
\end{align*}
Since the two $\alpha$-values that appear on the right hand side of above equation are given by the formulas in Lemma~\ref{alpha-lemma}, we obtain the desired result.  
\item The relative interior $\relint(F)$ of $F$ has no intersection with $b \Delta_{n-1}$: There exists a unique face $F'$ of $a C_n$ such that $\relint(F) \subseteq \relint(F')$. One checks that  $\dim(F) = \dim(F')$ and
  \begin{align*}
\fcone^p(F, P_n(a,b))\ =\ \fcone^p(F', a C_n)\ =\ \fcone^p(1/a \cdot F', C_n)\,.
  \end{align*}
Now, the formula follows again from Lemma~\ref{alpha-lemma}.
\end{enumerate}

\end{proof}

\subsection{Proof of the first part of  Proposition~\ref{main-prop}}

\begin{proposition}
The polytope $P_n(a,b)$ (for $a > b$) has positive BV-$\alpha$-values for $n \le 6$.
\label{crucial-prop}
\end{proposition}

\begin{proof}
Considering Lemma~\ref{crucial-lemma}, we see that as $2^{k-n} > 0$ it remains to compute $\alpha(F, P_n(a,b))$ for $k$-dimensional faces of $b \Delta_{n-1}$ for $n \le 6$.

\bgroup
\def\arraystretch{1.5}

\centering

\begin{table}[h]
\centering
\begin{tabular}{c|ccccccccc}
  & k &0 & 1                        & 2                       & 3               & 4             & 5             & 6           & 7 \\ \hline
n & &  &                          &                         &                 &               &               &             &   \\
1 &   &$\frac{1}{2}$& 1                        &                         &                 &               &               &             &   \\
2 &  &$\frac{1}{8}$& $\frac{1}{2}$             & 1                       &                 &               &               &             &   \\
3 &  &$\frac{1}{24}$& $\frac{5}{36}$           & $\frac{1}{2}$             & 1               &               &               &             &   \\
4 & & $\frac{1}{64}$& $\frac{1}{24}$            & $\frac{7}{48}$           & $\frac{1}{2} $    & 1             &               &             &   \\
5 &   &$\frac{1}{160}$& $\frac{9}{800} $          & $\frac{1}{24}   $         & $\frac{3}{20}$   & $\frac{1}{2} $  & 1             &             &   \\
6 &   &$\frac{1}{384}$& $\frac{1}{720}$           & $\frac{127}{14400} $      & $\frac{1}{24} $   & $\frac{11}{72}$ & $\frac{1}{2} $  & 1           &   \\
7 &  &$\frac{1}{896}$& $-\frac{5}{3136}$ & $-\frac{1}{800}$ & $\frac{61}{8400}$ & $\frac{1}{24} $ & $\frac{13}{84}$ & $\frac{1}{2}$ & 1
\end{tabular}
\caption{BV-$\alpha$-values for $k$-dimensional faces of type (i) in $P_n(a,b)$}
\label{tab:alpha1}
\end{table}
\egroup

One observes that all $\alpha$-values in the table are positive for $n \le 6$.
\end{proof}

We can now conclude the proof of Proposition~\ref{main-prop}.

\begin{corollary}
The polytope $Q_n(a,b)$ (for $a > 2b$) has positive BV-$\alpha$-values for $n \le 6$.
\end{corollary}

\begin{proof}
As the normal cone of any face of $Q_n(a,b)$ (with $a > 2b$) is isomorphic via an orthogonal unimodular transformation to that of a face of $P_n(2,1)$ (any `missing corner' of $Q_n(a,b)$ is a rotated $b S_n$), again by property (P4) of the BV-$\alpha$-valuation, they have the same BV-$\alpha$-values. 
\end{proof}

\subsection{BV-$\alpha$-positivity versus Ehrhart-positivity}

Even though the BV-$\alpha$-values start to have negative values at $n =7,$ 
the Ehrhart polynomials of $P_n(a,b)$ always have positive coefficients as we show in the lemma below. Hence, `BV-$\alpha$-positivity' is strictly stronger than `Ehrhart-positivity'.
\begin{lemma}\label{counter}
Suppose $P$ is a $n$-dimensional polytope that has the same normal fan as $P_n(a,b)$. Then $i(P,t)$ has positive coefficients.
\end{lemma}

\begin{proof} For any two polynomials $f$ and $g,$ we write $f \ge g$ if $[t^k] f \ge [t^k]g$ for each $k.$ So we need to show that $i(P,t) > 0.$
  One sees that $P$ is obtained from a rectangular box $[0,a_1]\times[0,a_2] \times \cdots \times [0,a_n]$ by chiseling off the origin at some distance $b,$ where $a_1,\dots, a_n, b$ are some positive integers satisfying $a_i > b$ for each $i.$ Therefore, letting $a = \min(a_i : 1 \le i \le n),$ we get
  \begin{align*}
    i(P,t)\ &=\ \prod_{i=1}^n (a_it+1) - \binom{bt+n}{n} + \binom{b t+n-1}{n-1}\\ &\ge\ (at+1)^n - \binom{bt+n-1}{n} \\
    &=\ \sum_{k=0}^n t^k \left( a^k \binom{n}{k} - b^k \cdot \frac{1}{n} \sum_{S: \text{$(k-1)$-subset of $[n-1]$}} \frac{1}{\prod_{s \in S} s}  \right)\\ &\ge\ \sum_{k=0}^n t^k \left( a^k \binom{n}{k} - b^k \cdot \frac{k^2}{n^2} \cdot \binom{n}{k} \cdot \frac{1}{k!} \right). 
  \end{align*}
  As $a > b$ and $\frac{k^2}{n^2 k!} \le 1$ for $0 \le k \le n$, the last expression is positive.
\end{proof}

\section*{Acknowledgements}
Part of this work was done when FC and BN met at the Fields Institute where BN took part in the Thematic Program Combinatorial Algebraic Geometry. BN would like to thank the Fields Institute for financial support and the organizers of the program for the invitation.

\end{document}